\documentclass{article}

\usepackage{amsmath,amsfonts,amsthm,amstext,amssymb,amscd,amsbsy}

\usepackage{indentfirst}
\usepackage[dvips]{color}
\usepackage{color}
\usepackage{epsfig}
\usepackage{hyperref}

\usepackage{lscape}
\usepackage{pdflscape}

\newtheorem{theorem}{Theorem}[section]
\newtheorem{proposition}[theorem]{Proposition}

\newtheorem{corollary}[theorem]{Corollary}

\newtheorem{remark}[theorem]{Remark}
\newtheorem{example}[theorem]{Example}

\newcommand{\CC}{\mathbb C}
\newcommand{\RR}{\mathbb R}

\newcommand{\mf}{\mathfrak}
\newcommand{\ad}{\mbox{ad}}
\newcommand{\F}{\mathbb{F}_{\Theta}}
\newcommand{\Pt}{R_{\Theta}}
\newcommand{\p}{\mf{p}}
\newcommand{\m}{\mf{m}}
\begin{document} 
\title{Invariant Einstein metrics on generalized flag manifolds of $Sp(n)$ and $SO(2n)$}

\author{Luciana Aparecida Alves \thanks{Federal University of Uberl\^andia, luciana.postigo@gmail.com} and Neiton Pereira da Silva \thanks{Federal University of Uberl\^andia, neitonps@gmail.com}}
\maketitle




\begin{abstract}
 
It is well known that the Einstein equation on a Riemannian flag manifold $(G/K,g)$ reduces to an algebraic system if $g$ is a $G$-invariant metric. In this paper we
obtain explicitly new invariant Einstein metrics on generalized flag manifolds of $Sp(n)$ and $SO(2n)$; and we compute the Einstein system for generalized flag manifolds of type $Sp(n)$. We also consider the
isometric problem for these Einstein metrics.

\end{abstract}

\noindent Mathematics Subject Classifications 53C25, 53C30, 14M17,
14M15, 22E46.  \newline


\section{Introduction}

A Riemannian manifold $(M,g)$ is called Einstein manifold if its Ricci tensor $Ric(g)$ satisfies the Einstein equation $Ric(g)=cg$, for some real constant $c$. 
The study of Einstein manifold is related with several areas of mathematics and has important applications on physics.(see \cite{Besse}, for example). 

Let $G$ be a connected compact semisimple Lie group and  $G/K$ a flag manifold, where $K$ is the centralizer of a torus in $G$. It is well known that the Einstein equation of a $G$-invariant (or simply invariant) metric $g$ on a flag manifold $G/K$ reduces to an (complicated in most cases) algebraic system. It is also known that $G/K$ admits an invariant K\"ahler Einstein metric associated to the canonical complex structure, see \cite{BH}. The problem of determining invariant Einstein metrics non K\"ahler has been studied by several authors, see for example \cite{Arv art}, \cite{art 1}, \cite{kimura}, \cite{Sakane} and \cite{w-z}.

In the algebraic Einstein system for flag manifolds, the number of unknowns is equal the number of equations and it is determined by the amount summands in the isotropy representation. In this sense, several authors have approached the problem of finding new Einstein metrics considering flag manifold with few isotropy summands, see \cite{Sakane}, \cite{Da Silva} and \cite{Arv-Chry}. Recently Wang-Zhao obtained, in \cite{w-z}, new invariant Einstein metrics on certain generalized flag manifolds with six isotropy summands using a computational method. 

Few authors have obtained new invariant Einstein metrics on generalized flag manifolds with many isotropy summands. For instance, Arvanitoyeorgos presented new Einstein metrics on generalized flag manifolds of type $SU(n)$ and $SO(2n)$, see \cite{Arv art}. In \cite{Sakane}, Sakane obtained new invariant Einstein metrics on full flag manifolds of a classical Lie group. 

Bohm-Wang-Ziller  conjectured in \cite{Bohm-Wang-Ziller} that if $G/H$ is a compact homogeneous space whose isotropy re-\ presentation consists of pairwise inequivalent irreducible summands, e.g. when rank $G$ = rank $H$, then the algebraic Einstein equations have only finitely many real solutions. In particular, this problem is opened yet for flag manifolds.

In this paper, following the method used in \cite{Arv art}, we computed explicitly the Einstein equations for generalized flag manifolds of type $Sp(n)$. Our main results, which extend partially the works \cite{Arv art} and \cite{Sakane}, are:

\textbf{Theorem A}
\emph{The flag manifold $Sp(n)/ U\left(
m\right) \times\cdots\times U\left( m\right)$ admits at least two not K\"{a}hler and non isometric invariant Einstein metrics
 \begin{eqnarray*}
 f &=& g=1\\ \\
h &=& \frac{2(m+1)+(n-2m+2)m \pm \sqrt{\Delta}}{2\left[(n-m)m+m+1\right]},\\ \\
c &=&\dfrac{4m(n-2m+2)[mn-m^2+m+1]+(m+1)[6mn-4m^2+4]\mp 2(m+1)\sqrt{\Delta}}{16(n+1)\left[(n-m)m+m+1\right]},
\end{eqnarray*}
where $\Delta=m^2n^2-4(m^3+m)n+(4m^4-8m^3+8m^2-4)$, $n\geq 2m$ and $n m \geq 2 \left[ m^2+1+\sqrt{2(m^3+1)}\right]$.
}

\vspace*{1cm}

\textbf{Theorem B}
\emph{The family of flag manifolds $SO(2n)/U(m)^s$, $m>1$, admits at least two non K\"ahler Einstein metrics, given by
\begin{eqnarray*}
f&=&g=1\\ \\
 h&=&\dfrac{n+2(m-1)\pm \sqrt{\Delta}}{2(n-1)}\\ \\
c&=&\dfrac{(n-2m-2)(2n-m-1)\mp\sqrt{\Delta}}{8(n-1)^2}
\end{eqnarray*}
where $n=sm$ and $\Delta=n^2-4(m-1)n+4(m^2-1)>0$. Besides these Einstein metrics are non isometric.}

\vspace*{1cm}

This paper is organized as follows: In Section 2 we discuss the construction of flag manifolds of a complex simple Lie group, and we use Weyl basis to see these spaces as the quotient $U/K_{\Theta}$ of a semisimple compact Lie group
$U\subset G$ modulo the centralizer $K_{\Theta}$ of a torus in
$U$. In Section 3 we recall the description of invariant metrics and its Ricci tensor on flag manifolds. The problem of isometric and non isometric metrics is also treated in the Section 4. In Section 5, we 
prove our results solving explicitly the algebraic Einstein system with a specific restriction condition on the invariant metrics.


\section{Preliminaries}
\label{sec:1}
In this section we set up our notation and present the standard
theory of partial (or generalized) flag manifolds associated with semisimple
Lie algebras, see for example \cite{SM N}, \cite{nir e sofia}, for similar description of flag manifolds.

 Let $\mathfrak{g}$ be a finite-dimensional semisimple complex Lie
algebra and take a Lie group $G$ with Lie algebra $\mathfrak{g}$. Let
$\mathfrak{h}$ be a Cartan subalgebra. We denote by $R$ the system of
roots of $(\mathfrak{g},\mathfrak{h})$. A root $\alpha\in R$ is a linear
functional on $\mathfrak{g}$. It uniquely determines an element
$H_{\alpha}\in\mathfrak{h}$ by the Riesz representation $\alpha(X)=B(
X,H_\alpha)$, $X\in\mathfrak{g}$, with respect to the Killing form
$B(\cdot,\cdot)$ of $\mf{g}$. The Lie algebra $\mathfrak{g}$ has the following
decomposition
\[
\mathfrak{g}=\mathfrak{h}\oplus\sum_{\alpha\in R}\mathfrak{g}_{\alpha}
\]
where $\mathfrak{g}_{\alpha}$ is the one-dimensional root space corresponding to $\alpha$. Besides the eigenvectors $E_{\alpha}\in\mathfrak{g}_{\alpha}$ satisfy the following equation

\begin{equation}
\left[  E_{\alpha},E_{-\alpha}\right]  =B\left(
E_{\alpha},E_{-\alpha}\right) H_{\alpha}. \label{igualdade}
\end{equation}

We fix a system $\Sigma$ of simple roots of $R$ and denote by
$R^+$ and $R^{-}$ the corresponding set of positive and
negative roots, respectively. Let $\Theta\subset\Sigma$ be a
subset, define
\begin{align}
R_{\Theta}:=\langle\Theta\rangle\cap R\hspace{1.5cm}
\Pt^\pm:=\langle\Theta\rangle\cap R^{\pm}.\nonumber
\end{align}
We denote by $R_M:=R\setminus\Pt$ the complementary set of roots. Note that
\[
\p_{\Theta}:=\mf{h}\oplus\sum_{\alpha\in R^+}\mf{g}_{\alpha}\oplus\sum_{\alpha\in R_{\Theta}^-}\mf{g}_{\alpha}
\]
is a parabolic subalgebra, since it contains the Borel subalgebra
$\mf{b}^+=\mf{h}\oplus\sum\limits_{\alpha\in R^+}\mf{g}_{\alpha}$.

The partial flag manifold determined by the choice $\Theta\subset R$ is the
homogeneous space $\F=G/P_{\Theta}$, where $P_{\Theta}$ is the normalizer of
$\mf{p}_{\Theta}$ in $G$. In the special case $\Theta=\emptyset$, we obtain the \emph{full} (or maximal) flag manifold $\mathbb{F}=G/B$ associated with $R$, where
$B$ is the normalizer of the Borel subalgebra $\mf{b}^+=\mf{h}\oplus\sum\limits_{\alpha\in R^+}\mf{g}_{\alpha}$ in $G$. 

For further use, to each $\alpha\in R_M$, define the following sets
\begin{align}\label{sets}
\Pt(\alpha)&:=\left\{\phi\in \Pt:(\alpha+\phi)\in R\right\}\quad \text{and}\nonumber\\
R_M(\alpha)&:=\{\beta\in R_M:(\alpha+\beta)\in R_M\}.
\end{align}

Now we will discuss the construction of any flag manifold as the
quotient $U/K_{\Theta}$ of a semisimple compact Lie group
$U\subset G$ modulo the centralizer $K_{\Theta}$ of a torus in
$U$. We fix once and for all a Weyl base of $\mf{g}$ which amounts
to giving $X_\alpha\in\mf{g_{\alpha}}$, $H_{\alpha}\in\mf{h}$ with $\alpha\in R$, with the standard
properties:
\begin{equation}\label{base weyl}
\begin{tabular}
[c]{lll}
$B( X_{\alpha},X_{\beta}) =\left\{
\begin{array}
[c]{cc}%
1, & \alpha+\beta=0,\\
0, & \text{otherwise};
\end{array}
\right. $ &\hspace{-.5cm} & $\left[  X_{\alpha},X_{\beta}\right]  =\left\{
\begin{array}
[c]{cc}%
H_{\alpha}\in\mathfrak{h}, & \alpha+\beta=0,\\
N_{\alpha,\beta}X_{\alpha+\beta}, & \alpha+\beta\in R,\\
0, & \text{otherwise.}%
\end{array}
\right.  $
\end{tabular}
\end{equation}
The real numbers $N_{\alpha,\beta}$ are non-zero if and only if
$\alpha+\beta\in R$. Besides that it satisfies
$$
\left\{
\begin{array}
[c]{cc}%
 N_{\alpha,\beta}=-N_{-\alpha,-\beta}=-N_{\beta,\alpha},& \\
\hspace*{-1cm} N_{\alpha,\beta}=N_{\beta,\gamma}=N_{\gamma,\alpha},&\mbox{if}\quad\alpha+\beta+\gamma=0.
\end{array}
\right.
$$

We consider the following two-dimensional real spaces
$\mf{u}_{\alpha}=\text{span}_\mathbb{R}\{A_{\alpha},S_{\alpha}\}$, 
where $\hspace{0.3cm}A_{\alpha}=X_{\alpha}-X_{-\alpha}$
and $\hspace{0.3cm}S_{\alpha}=i(X_{\alpha}+X_{-\alpha})$, with
$\alpha\in R^{+}$.
Then the real Lie algebra
$\mf{u}=i\mf{h}_\mathbb{R}\oplus\sum\mf{u}_{\alpha}$, with $\alpha\in R^{+},$
is a compact real form of $\mf{g}$, where $\mf{h}_\mathbb{R}$ denotes the subspace of $\mf{h}$ spanned by $\{H_\alpha, \alpha\in R\}$.

Let $U=\exp\mf{u}$ be the compact real form of $G$ corresponding
to $\mf{u}$. By the restriction of the action of $G$ on $\F$, we
can see that $U$ acts transitively on $\F$ then $\F=U/K_{\Theta}$,
where $K_{\Theta}=P_{\Theta}\cap U$. The Lie algebra $\mf{k}_{\Theta}$ of $K_\Theta$ is the set of fixed points of the conjugation $\tau\colon X_{\alpha}\mapsto-X_{-\alpha}$ of $\mf{g}$ restricted to $\p_{\Theta}$
$$
\mf{k}_{\Theta}=\mf{u}\cap\p_{\Theta}=i\mf{h}_\mathbb{R}\oplus\sum_{\alpha\in\Pt^+}\mf{u}_{\alpha}.
$$

The tangent space of $\F=U/K_{\Theta}$ at the origin $o=eK_{\Theta}$ can be identified with the orthogonal complement (with respect to the Killing form) of $\mf{k}_{\Theta}$ in $\mf{u}$
$$
T_o\F=\m=\sum\limits_{\alpha\in R_{M}^{+}}\mf{u}_{\alpha}
$$
with $R_{M}^{+}=R_{M}\cap R^+$. Thus we have $\mf{u}=\mf{k}_{\Theta}\oplus\mf{m}$.

On the other hand, there exists a nice way to decompose the tangent space $\m$. It is known (see for example \cite{Alek e Perol} or \cite{Sie}) that $\F$ is a reductive homogeneous space, this means that the adjoint representation of $\mf{k}_{\Theta}$ and $K_{\Theta}$ leaves $\mf{m}$ invariant, i.e. $\ad(\mf{k}_{\Theta})\mf{m}\subset\mf{m}$. Thus we can decompose $\mf{m}$ into a sum of irreducible $\ad (\mf{k}_{\Theta})$ submodules $\mf{m}_i$ of the module $\mf{m}$:
\[
\m= \m_1\oplus\cdots\oplus\m_s.
\]

Now we will see how to obtain each irreducible $\ad (\mf{k}_{\Theta})$ submodules $\m_i$. By complexifying the Lie algebra of $K_{\Theta}$ we obtain
$$
\mf{k}_{\Theta}^{\mathbb{C}}=\mf{h}\oplus\sum_{\alpha\in\Pt}\mf{g}_{\alpha}.
$$
The adjoint representation of $\ad(\mf{k}_{\Theta}^{\mathbb{C}})$ of $\mf{k}_{\Theta}^{\mathbb{C}}$ leaves the complex tangent space $\m^{\mathbb{C}}$ invariant.
Let
$$
\mf{t}:=Z(\mf{k}_{\Theta}^\mathbb{C})\cap i\mf{h}_\mathbb{R}
$$
be the intersection of the center of the subalgebra
$\mf{k}_{\Theta}^\CC$ with $i\mf{h}_\RR$. According to \cite{Arv art}, we can write
$$
\mf{t}=\{H\in i\mf{h}_{\RR}:\alpha(H)=0,\,\text{for all} \,\alpha \in \Pt\}.
$$

Let $i\mf{h}_\RR^{\ast}$ and $\mf{t}^{\ast}$ be the dual vector
space of $i\mf{h}_\RR$ and $\mf{t}$, respectively, and consider the
map $k\colon i\mf{h}_\RR^{\ast}\longrightarrow\mf{t}^{\ast}$ given
by $k(\alpha)=\alpha|_{\mf{t}}$. The linear functionals of
$R_{\mf{t}}:=k(R_{M})$ are called \emph{t-roots}. Denote by $R_{\mf{t}}^+=k(R_M^+)$ the set of positive t-roots. There exists a 1-1 correspondence between positive t-roots and irreducible submodules of the adjoint representation of $\mf{k}_{\Theta}$, see \cite{Alek e Perol}. This correspondence is given by
\begin{equation*}
\xi\longleftrightarrow\mf{m}_{\xi}=\sum_{k(\alpha)=\xi}\mf{u}_{\alpha}
\end{equation*}
with $\xi\in R_{\mf{t}}^+$. Besides these submodules are inequivalents. Hence the tangent space can be decomposed as follows
\begin{equation*}\label{decomp}
\m=\m_{\xi_1}\oplus\cdots\oplus\m_{\xi_s}
\end{equation*}
where $R_t^+=\lbrace \xi_1, \ldots, \xi_s\rbrace$.

\section{Invariant metrics and Ricci tensor on $\F$}

A Riemannian invariant metric on $\F$ is completely determined by a real inner product $g\left(\cdot,\cdot\right)$ on $\mathfrak{m}=T_{o}\F$ which is invariant by the adjoint action of $\mf{k}_{\Theta}$. Besides that any real inner product $\ad(\mf{k}_{\Theta})$-invariant on $\mf{m}$ has the form

\begin{equation}\label{inner product}
g\left(\cdot,\cdot\right)=-\lambda_1 B\left(\cdot,\cdot\right)|_{\mathfrak{m}_{1}\times\mathfrak{m}_{1}}-\cdots
-\lambda_s B\left(\cdot,\cdot\right)|_{\mathfrak{m}_{s}\times\mathfrak{m}_{s}}
\end{equation}\\
where ${\mathfrak{m}}_{i}=\mf{m}_{\xi_i}$ and $\lambda_i=\lambda_{\xi_i}>0$ with $\xi_{i}\in R^{+}_{\mathfrak{t}}$, for $i=1,\ldots,s$. So any invariant Riemannian metric on $\F$ is determined by $|R_{\mf{t}}^+|$ positive parameters. We will call an inner product defined by (\ref{inner product}) as an invariant metric on $\F$.

In a similar way, the Ricci tensor $Ric_g(\cdot.\cdot)$ of a invariant metric on $\F$ depends on $|R_{\mf{t}}^+|$ parameters. Actually, it has the form
\[
Ric_g\left(\cdot,\cdot\right)=-r_ 1\lambda_1 B\left(\cdot,\cdot\right)|_{\mathfrak{m}_{1}\times\mathfrak{m}_{1}}-\cdots
-r_s\lambda_s B\left(\cdot,\cdot\right)|_{\mathfrak{m}_{s}\times\mathfrak{m}_{s}}
\]
where $r_i$ are constants.
Thus an invariant metric $g$ on $\F$ is Einstein iff $r_1=\cdots=r_s$. The next result shows a way to compute the components of the Ricci tensor by means of vectors of Weyl base.

\begin{proposition}(\cite{Arv art})
\label{Ricci} The Ricci tensor for an invariant metric $g$  on $\F$ is
given by
\begin{align}\label{Ric}
Ric\left(  X_{\alpha},X_{\beta}\right)
&=0,\quad\alpha,\beta\in R
_{M},\quad\alpha+\beta\notin R_{M},\nonumber\\
& \\
Ric(X_{\alpha},X_{-\alpha})&=B\left(  \alpha,\alpha\right) +\sum
_{\substack{\phi\in\Pt\\\alpha+\phi\in R}}N_{\alpha,\phi}^{2}+\frac{1}
{4}\sum_{\substack{\beta\in R_{M}\\\alpha+\beta\in R_{M}}}\frac{N_{\alpha,\beta}^{2}}{\lambda_{\alpha+\beta
}\lambda_{\beta}}\left( \lambda_{\alpha}^{2}-\left(
\lambda_{\alpha+\beta}-\lambda_{\beta}\right) ^{2}\right)\nonumber.
\end{align}
\end{proposition}
Since $Ric(\kappa
g)=Ric(g)$ ($\kappa\in\mathbb{R}$), one can normalize the Einstein equation $Ric(g)=c\cdot g$
choosing an appropriate value for $c$ or for some $\lambda_{\alpha}$.
\begin{remark}\label{remark}
Although (\ref{Ric}) is not in terms of t-roots, if $\alpha, \beta\in R_M$ are two different roots that determine the same t-root, i.e.  $k(\alpha)=k(\beta)$, then $\lambda_{\alpha}=\lambda_{\beta}$ and $Ric(X_{\alpha},X_{-\alpha})=Ric(X_{\beta},X_{-\beta})$.
\end{remark}

In \cite{Park-Sakane}, Park-Sakane computed the Ricci tensor in a similar way. In their  formula appears the dimension $d_i$ of each irreducible submodules $\mf{m_i}$, while (equivalently) the equation (\ref{Ric}) depends on the amounts of factors $U(n_i)$ in the isotropy subgroup $K$.  Actually Park-Sakane formula is very useful when one wants to describe the Ricci tensor on homogeneous spaces with few isotropy summands or maximal flag manifolds (see for example \cite{w-z}, \cite{Arv-Chry} \cite{Sakane}). The advantage of using (\ref{Ric}) is that we can examine at once the Einstein equation for different families of flag manifolds, of the same type, in terms of the size and the amounts of $U(n)$-factors in the isotropy subgroup $K$. We will use Proposition \ref{Ricci} to complete the list of the algebraic Einstein system for all generalized flag manifolds of classical Lie groups.

\section{Isometric and non isometric metrics} 

We discuss the problem of determining if two invariant Einstein metrics on $\F$ are isometric or non isometric. 

Let $\F$ be a flag manifold with isotropy decomposition 
$$
\m=\m_{1}\oplus\cdots\oplus\m_{s}
$$
and denote by $d=\dim \m=\dim\F$ and $d_i=\dim\m_i$, $i=1,\dots,s$. Given an invariant Einstein metric $g=(\lambda_1, \dots, \lambda_s)$ on $\F$, its volume is given by $V_g=\prod\limits_{i=1}^{s} \lambda_i^{d_i}$.  Consider the scale
$$
H_g=V^{1/d}S_g,
$$ 
where $S_g=\sum\limits_{i=1}^{s}d_i r_i$ is the scalar curvature of $g$,  $V=V_g/V_B$ and $V_B$ denotes the volume of the normal metric induced by the negative of the Killing form in $U$ (compact real form of $G$).  We normalize $V_B=1$, then $H_g=V_g^{1/d}S_g$. It is known that $H_g$ is a scale invariant under a common scaling of the parameter $\lambda_i$ (see \cite{Arv-Chry} or \cite{w-z}).

If $g_1$ and $g_2$ are two invariant Einstein metrics on $\F$ are isometric then $H_{g_1}=H_{g_2}$. Thus if $H_{g_1}\neq H_{g_2}$ then $g_1$ and $g_2$ are non isometric. In general, to determine if two invariant Einstein metrics are isometric is not a trivial problem (see for example \cite{Arv-Chry-Sakane}).

Now we note that if $g$ is an invariant Einstein metric then $S_g=c\cdot d$, where $c$ is the Einstein constant from $Ric_g(\cdot,\cdot)=cg(\cdot,\cdot)$. Besides, if $g$ has volume $V_g$ then $\widehat{g}= \dfrac{1}{V_g^{1/d}}g$ has volume $V_{\widehat{g}}=1$ and in this case $H_{\widehat{g}}=c d$, 
since 
$Ric_{\widehat{g}}(\cdot,\cdot)=Ric_{g}(\cdot,\cdot)=cg(\cdot,\cdot)$. So if $g_1$ and $g_2$ are two invariant Einstein metrics with different Einstein constants $c_1$ and $c_2$, then $g_1$ and $g_2$ are non isometric.

\section{Proof of Theorem A}


In this section we consider flag manifolds of the form
$Sp(n)/ U\left(
n_{1}\right) \times\cdots\times U\left( n_{s}\right)$, where $n\geq 3$ and $n=\sum n_{i}$. 


The next result was obtained in a different way in \cite{Graev}, we proved this theorem following the method used in \cite{Arv art} with the aim of introducing the notation.

\begin{theorem}\label{t-roots properties Cn}
The set $R_{\mathfrak{t}}$ of t-roots corresponding to the flag manifolds $Sp(n)/ U\left(
n_{1}\right) \times\cdots\times U\left( n_{s}\right)$ is a system of roots of type $C_{s}$.
\end{theorem}

\begin{proof}
A Cartan subalgebra of $\mathfrak{sp}(n,\mathbb{C})$ consists in
taking matrices of the form
\begin{equation}
\mf{h}=
\begin{pmatrix}
\Lambda & 0\\
0 & -\Lambda
\end{pmatrix}\label{alg cartan Cl}
\end{equation}
where $\Lambda=\mathrm{diag}(\varepsilon_{1},\ldots,\varepsilon_{n})$, $\varepsilon_{i}%
\in\mathbb{C}$. Following the notation of \cite{Arv art}, we will denote the linear functional $\mf{h}\mapsto\pm 2\varepsilon_{i}$ and $\mf{h}\mapsto\pm(\varepsilon_{i}\pm\varepsilon_{j})$ by $\pm 2\varepsilon_{i}$ and $\pm(\varepsilon_{i}\pm\varepsilon_{j})$ respectively. Thus the root system is
\begin{equation}
R=\{  \pm\left(  \varepsilon_{i}\pm\varepsilon_{j}\right);\,1\leq i<j\leq
n\}  \cup\{ \pm2\varepsilon_{i};\,1\leq i\leq
n\}.\label{Pi Cl}
\end{equation}
The root system for the subalgebra $\mathfrak{k}_{\Theta}^{\mathbb{C}
}=\mathfrak{sl}\left(  n_{1},\mathbb{C}\right)
\times\cdots\times\mathfrak{sl}\left( n_{s},\mathbb{C}\right)$ is
given by
\[
R_{\Theta} =\{  \pm\left(  \varepsilon_{a}^{i}-\varepsilon_{b}
^{i}\right)  ;\,1\leq a<b\leq n_{i},1\leq i\leq s \}.
\]
Then
\[
R_{M}=\{  \pm (  \varepsilon_{a}^{i}\pm\varepsilon_{b}
^{j});\,1\leq i<j\leq s\} \\
\cup\{  \pm\left(
\varepsilon_{a}^{i}+\varepsilon_{b}^{i}\right) ;\,1\leq i\leq
s,1\leq a< b\leq n_{i}\}
\]
and the algebra $\mf{t}$ has the form
\[
\mf{t}=
\begin{pmatrix}
\Lambda & 0\\
0 & -\Lambda
\end{pmatrix}
\]
with
$
\Lambda
=\mathrm{diag}\left(\varepsilon^{1}_{n_{1}},\ldots,\varepsilon^{1}_{n_{1}},\varepsilon^{2}_{n_{2}},\ldots,\varepsilon^{3}_{n_{3}},\ldots,
\varepsilon^{s}_{n_{s}},\ldots,\varepsilon^{s}_{n_{s}}\right)
$.
Here each $\varepsilon^{i}_{n_{i}}$ appears exactly $n_{i}$
times, $i=1,\ldots,s$. So restricting the roots of
$R_{M}$ in $\mf{t}$, and using the notation
$\delta_{i}=k(\varepsilon^{i}_{a})$, we obtain the t-root set:
\[
R_{\mathfrak{t}}=\{\pm\left(\delta_{i}-\delta_{j}\right),\pm\left(\delta_{i}+\delta_{j}\right);\,1\leq
i<j\leq s\}\cup\{\pm 2\delta_{i};\,1\leq i\leq
s\}.
\]
Note that
$k(\varepsilon_{a}^{i}+\varepsilon_{b}^{i})
=k(2\varepsilon_{a}^{i})$, $ 1\leq i\leq s$. In particular, there exist $s^{2}$ positive t-roots.
\end{proof}

Next we are going to compute the Einstein system for a invariant metric on $Sp(n)/ U\left(
n_{1}\right) \times\cdots\times U\left( n_{s}\right)$. The Killing form of
$\mathfrak{sp(n)}$ is given by
$$B\left( X,Y\right) =2\left( n+1\right)\mathrm{tr}(XY),$$ 
and
$$ B( \alpha,\alpha)=\left\{
\begin{array}
[c]{cc}
1/(n+1), & \text{if} \quad \alpha=\pm2\varepsilon_{i},\\
1/2(n+1), & \text{if} \quad \alpha= \pm ( \varepsilon
_{i}\pm\varepsilon_{j}),\quad 1\leq i<j\leq
n.
\end{array}
\right.
$$
Then the eigenvectors $X_{\alpha}\in \mf{g}_{\alpha}$ satisfying (\ref{base weyl})  are
\begin{align*}
&X_{\pm\left(\varepsilon_{i}-\varepsilon_{j}\right)}=\pm\frac
{1}{2\sqrt{n+1}}E_{\pm\left(
\varepsilon_{i}-\varepsilon_{j}\right)},\quad X_{\pm\left(
\varepsilon_{i}+\varepsilon_{j}\right)}=\pm\frac{1}
{2\sqrt{n+1}}E_{\pm\left(  \varepsilon_{i}+\varepsilon_{j}\right)
},\quad 1\leq i<j\leq n;&\\ \\&
X_{\pm2\varepsilon_{i}}=\pm\frac{1}{\sqrt{2\left(  n+1\right)  }}
E_{\pm2\varepsilon_{i}},\quad 1\leq i\leq n,
\end{align*}
where $E_{\alpha}$ denotes the canonical eigenvectors of $\mf{g}_{\alpha}$. It is convenient to use the following notation
\begin{align*}
E_{ab}^{ij}  &  =X_{\varepsilon_{a}^{i}-\varepsilon_{b}^{j}},\quad
F_{ab}^{ij}=X_{\varepsilon_{a}^{i}+\varepsilon_{b}^{j}},\ F_{-ab}^{ij}=X_{-\left(\varepsilon_{a}^{i}+\varepsilon_{b}^{j}\right)
},\quad 1\leq i<j\leq s;\\
F_{ab}^{i}  &  =X_{\varepsilon_{a}^{i}+\varepsilon_{b}^{i}},\ F_{-ab}^{-i}=X_{-\left(  \varepsilon_{a}^{i}+\varepsilon_{b}^{i}\right)
},\quad 1\leq
a\neq b\leq n_{i};\\ G_{a}^{i} & =X_{2\varepsilon_{a}^{i}},\quad
G_{-a}^{-i}=X_{-2\varepsilon_{a}^{i}},\text{ }1\leq i\leq s.\nonumber
\end{align*}
An invariant metric on
$\mathbb{F}_{C}(n_{1},\ldots,n_{s})$ will be denoted by
\begin{align}\label{inv metric Cl}
g_{ij}  &  =g\left(  E_{ab}^{ij},E_{ba}^{ji}\right), \quad f_{ij}=g\left(
F_{ab}^{ij},F_{-ab}^{-ij}\right)  ,\quad 1\leq i<j\leq s;\\
h_{i}  &  =g\left(  G_{a}^{i},G_{-a}^{-i}\right),\quad l_{i}=g\left( F_{ab}^{i},F_{-ab}^{-i}\right)  ,\quad 1\leq i\leq s.\nonumber
\end{align}
Since
$k(\varepsilon_{a}^{i}+\varepsilon_{b}^{i})
=k(2\varepsilon_{a}^{i})=2k(\varepsilon_{a}^{i})$ it follows
\[
l_{i}=h_{i},\quad 1\leq i\leq s,
\]
by remark \ref{remark}.

Considering short and long roots of $\mathfrak{sp}(n)$, one can see that the square of structural constants are given by
\begin{align*}
& N_{\left(  \varepsilon_{i}+\varepsilon_{j}\right)  ,\pm\left(
\varepsilon _{i}-\varepsilon_{j}\right)
}^{2}=N_{\pm2\varepsilon_{j},\left(
\varepsilon_{i}\mp\varepsilon_{j}\right)
}^{2}=N_{-2\varepsilon_{i},\left(
\varepsilon_{i}\pm\varepsilon_{j}\right)  }^{2}=\frac{1}{2\left(
n+1\right) }, \quad i\neq j;\\
& N_{\left(  \varepsilon_{i}-\varepsilon_{j}\right)
,\alpha}^{2}=N_{\left( \varepsilon_{i}+\varepsilon_{j}\right)
,\beta}^{2}=\frac{1}{4\left( n+1\right)}
\end{align*}
if $\alpha\in\{  \left(
\varepsilon_{k}-\varepsilon_{i}\right)  ,\left(
\varepsilon_{j}-\varepsilon_{k}\right)  ,\left(  \varepsilon_{j}
+\varepsilon_{l}\right)  ,-\left(
\varepsilon_{i}+\varepsilon_{p}\right): \,p\neq i;l\neq
j; k\neq i,j; i\neq j\}$ and
$\beta\in\{  -\left(  \varepsilon_{i}+\varepsilon_{k}\right)
,-\left(  \varepsilon_{j}+\varepsilon_{l}\right):\, k\neq
j;l\neq i\}$.

In the next table we compute $R_{\Theta}(\alpha)$ and $R_M(\alpha)$ for each $\alpha\in R_M$.

\begin{landscape}
\begin{small}
\begin{table}
\caption{The sets $R_{\Theta}\left(\alpha\right)$ and $R_{M}\left(\alpha\right)$ for $\mathbb{F}_{C}(n_{1},\ldots,n_{s})$}
\begin{tabular}{ccccccc}
  \hline
  $\alpha\in R_{M}$ &  &  &$R_{\Theta}\left(\alpha\right)$ is the union of &  & &$R_{M}\left(\alpha\right)$ is the union of  \\
  \hline
 &  &  &  &  &  &  \\
$\varepsilon_{c} ^{k}-\varepsilon_{d}^{t}$ &  &  &$\{ \varepsilon_{a}^{k}-\varepsilon_{c}^{k}:\scriptsize{1\leq a\leq
n_{k},a\neq c}\}$  &  &  & $ \{  \left(  \varepsilon_{d}^{t}-\varepsilon_{a}^{i}\right)
,\left(
\varepsilon_{a}^{i}-\varepsilon_{c}^{k}\right)  ,\left(  \varepsilon_{a}%
^{i}+\varepsilon_{d}^{t}\right)  ,-\left(
\varepsilon_{a}^{i}+\varepsilon _{c}^{k}\right)\}$\\
\scriptsize{$1\leq k<t\leq s$}  &  & &$\{
\varepsilon_{d}^{t}-\varepsilon_{a}^{t}:\scriptsize{1\leq a\leq n_{t},a\neq
d}\} $   &  &  & \footnotesize{$1\leq i\leq
s,i\neq k,t\quad\text{and}\quad 1\leq a\leq n_{i}$} \\
&  &  &  &  &  & \\
 &  &  &  &  &  & $\{  \left(  \varepsilon_{a}^{k}+\varepsilon_{d}^{t}\right)
,-\left( \varepsilon_{a}^{k}+\varepsilon_{c}^{k}\right):\scriptsize{1\leq
a\leq n_{k}}\}$   \\
&  &  &  &  &  &$\{  -\left(
\varepsilon_{a}^{t}+\varepsilon_{c}^{k}\right)  ,\left(
\varepsilon_{a}^{t}+\varepsilon_{d}^{t}\right):\scriptsize{1\leq a\leq
n_{t}}\}$  \\
 \hline
 &  &  &  &  &  &  \\
$\varepsilon_{c}^{k}+\varepsilon_{d}^{t}$ &  &  &$\{ \varepsilon_{a}^{k}-\varepsilon_{c}^{k}:\scriptsize{1\leq a\leq
n_{k},a\neq c}\}$ &  &  & $ \{  \left(  \varepsilon_{a}^{i}-\varepsilon_{c}^{k}\right)
,\left(
\varepsilon_{a}^{i}-\varepsilon_{d}^{t}\right)  ,-\left(  \varepsilon_{a}%
^{i}+\varepsilon_{c}^{k}\right)  ,-\left(
\varepsilon_{a}^{i}+\varepsilon
_{d}^{t}\right)\}$  \\
\scriptsize{$1\leq k<t\leq s$} &  &  &$\{
\varepsilon_{a}^{t}-\varepsilon_{d}^{t}:\scriptsize{1\leq a\leq n_{t},a\neq
d}\}$  &  &  & \footnotesize{$ 1\leq i\leq s,i\neq k,t;1\leq a\leq n_{i}$} \\
 &  &  &   &  &  &  \\
 &  &  &  &  &  & $ \{  \left(  \varepsilon_{a}^{t}-\varepsilon_{c}^{k}\right),-\left( \varepsilon_{a}^{t}+\varepsilon_{d}^{t}\right):\scriptsize{1\leq a\leq n_{t}}\}$  \\
 &  &  &  &  &  & $\{  \left(
\varepsilon_{a}^{k}-\varepsilon_{d}^{t}\right)  ,-\left(
\varepsilon_{a}^{k}+\varepsilon_{c}^{k}\right):\scriptsize{1\leq a\leq
n_{k}}\} $  \\
 \hline
 &  &  &  &  &  &  \\
$\varepsilon_{c}^{k}+\varepsilon_{d}^{k}$ &  &  & $\{ \left(  \varepsilon_{a}^{k}-\varepsilon_{c}^{k}\right)
,\left( \varepsilon_{a}^{k}-\varepsilon_{d}^{k}\right)\}$  &  &  & $\{ \left(  \varepsilon_{a}^{i}-\varepsilon_{c}^{k}\right)
,\left(
\varepsilon_{a}^{i}-\varepsilon_{d}^{k}\right)  ,-\left(  \varepsilon_{a}%
^{i}+\varepsilon_{c}^{k}\right)  ,-\left(
\varepsilon_{a}^{i}+\varepsilon _{d}^{k}\right)\}$ \\
\scriptsize{$1\leq k\leq s$} &  &  &\footnotesize{$1\leq
a\leq n_{k};a\neq
c,d$}  &  &  &\footnotesize{$1\leq i\leq
s;i\neq k$}  \\
\scriptsize{$1\leq c<d\leq n_{k}$} &  &  &  &  &  &  \\
 &  &  & $\{  \pm\left(  \varepsilon_{c}^{k}-\varepsilon_{d}^{k}\right)  \}$  &  &  &  \\
 \hline
&  &  &  &  &  &  \\
$2\varepsilon_{c}^{k}$&  &  & $\{  \varepsilon_{a}
^{k}-\varepsilon_{c}^{k}:\scriptsize{1\leq a\leq n_{k};a\neq c}\}$ &  &  &  $\{  \left(
\varepsilon
_{a}^{i}-\varepsilon_{c}^{k}\right)  ,-\left(  \varepsilon_{a}^{i}%
+\varepsilon_{c}^{k}\right):\scriptsize{1\leq i\leq s;i\neq k}\}$\\
\scriptsize{$1\leq k\leq s$}&  &  &  &  &  &  \\
&  &  &  &  &  &  \\
\hline
 \end{tabular}
 \end{table}
 \end{small}
 \end{landscape}

Using Proposition \ref{Ricci} we obtain the following result.

\begin{proposition}\label{Einstein equation Cn 1}
The Einstein equation for an invariant metric on
$Sp(n)/ U\left(
n_{1}\right) \times\cdots\times U\left( n_{s}\right)$ reduces to an
algebraic system where the number of unknowns and equations is $s^2$, given by

\begin{align*}
&\dfrac{1}{8(n+1)}\left\lbrace 2(n_{k}+n_{t})+\frac{( n_{k}+1) }{h_{k}f_{kt}}\left(
g_{kt}^{2}-\left( h_{k}-f_{kt}\right)  ^{2}\right)+\frac{\left(  n_{t}+1\right)  }{h_{t}f_{kt}}\left(  g_{kt}^{2}-\left(  h_{t}-f_{kt}\right)  ^{2}\right)\right. \\ \\
&\left. +\sum_{i\neq k,t}^{s}\frac{n_{i}}{g_{ik}g_{it}
}\left(  g_{kt}^{2}-\left(  g_{ik}-g_{it}\right)  ^{2}\right)+
\sum_{i\neq k,t}^{s}\frac{n_{i}}{f_{ik}f_{it}
}\left(  g_{kt}^{2}-\left(  f_{ik}-f_{it}\right)  ^{2}\right)\right\rbrace 
=c g_{kt},\,\quad 1\leq k\neq t\leq s;
\end{align*}

\begin{align*}
&\dfrac{1}{8(n+1)}\left\lbrace 2(n_{k}+n_{t})+\frac{\left( n_{k}+1\right) }{h_{k}g_{kt}}\left(
f_{kt}^{2}-\left( h_{k}-g_{kt}\right)  ^{2}\right)
+\frac{\left(  n_{t}+1\right)  }{h_{t}g_{kt}%
}\left(  f_{kt}^{2}-\left(  h_{t}-g_{kt}\right)  ^{2}\right)\right. \\ \\
&+\left. \sum_{i\neq k,t}^{s}\frac{n_{i}}{f_{it}g_{ik}%
}\left(  f_{kt}^{2}-\left(  f_{it}-g_{ik}\right)  ^{2}\right)
+\sum_{i\neq k,t}^{s}\frac{n_{i}}{f_{ik}g_{it}
}\left(  f_{kt}^{2}-\left(  f_{ik}-g_{it}\right)  ^{2}\right)\right\rbrace 
=cf_{kt},\quad 1\leq k\neq t\leq s;\\ \\
&\dfrac{1}{8(n+1)}\left\lbrace 4(n_{k}+1)+2\sum_{i\neq k}^{s}\frac{n_{i}}{f_{ik}g_{ik}%
}\left(  h_{k}^{2}-\left(  f_{ik}-g_{ik}\right)  ^{2}\right)\right\rbrace =c h_{k},\quad 1\leq k\leq s.
\end{align*}

\end{proposition}
Now we consider the flag manifold

$$
Sp(n)/ U\left(
m\right) \times\cdots\times U\left( m\right),
$$
where $n=ms$. By the Proposition \ref{Einstein equation Cn 1}, the Einstein equations is given by 

\begin{align*}
&\dfrac{1}{8(n+1)}\left\lbrace 4m+\frac{( m+1) }{h_{k}f_{kt}}\left(
g_{kt}^{2}-\left( h_{k}-f_{kt}\right)  ^{2}\right)+\frac{\left(  m+1\right)  }{h_{t}f_{kt}}\left(  g_{kt}^{2}-\left(  h_{t}-f_{kt}\right)  ^{2}\right)\right. \\ \\
&+\left. \sum_{i\neq k,t}^{s}\frac{m}{g_{ik}g_{it}
}\left(  g_{kt}^{2}-\left(  g_{ik}-g_{it}\right)  ^{2}\right)+
\sum_{i\neq k,t}^{s}\frac{m}{f_{ik}f_{it}
}\left(  g_{kt}^{2}-\left(  f_{ik}-f_{it}\right)  ^{2}\right)\right\rbrace 
=cg_{kt},\,\quad 1\leq k\neq t\leq s;
\end{align*}

\begin{align*}
&\dfrac{1}{8(n+1)}\left\lbrace 4m+\frac{\left( m+1\right) }{h_{k}g_{kt}}\left(
f_{kt}^{2}-\left( h_{k}-g_{kt}\right)  ^{2}\right)
+\frac{\left(  m+1\right)  }{h_{t}g_{kt}%
}\left(  f_{kt}^{2}-\left(  h_{t}-g_{kt}\right)  ^{2}\right)\right. \\ \\
&+\left. \sum_{i\neq k,t}^{s}\frac{m}{f_{it}g_{ik}%
}\left(  f_{kt}^{2}-\left(  f_{it}-g_{ik}\right)  ^{2}\right)
+\sum_{i\neq k,t}^{s}\frac{m}{f_{ik}g_{it}
}\left(  f_{kt}^{2}-\left(  f_{ik}-g_{it}\right)  ^{2}\right)\right\rbrace 
=c f_{kt},\quad 1\leq k\neq t\leq s;\\ \\
&\dfrac{1}{8(n+1)}\left\lbrace 4(m+1)+2\sum_{i\neq k}^{s}\frac{m}{f_{ik}g_{ik}%
}\left(  h_{k}^{2}-\left(  f_{ik}-g_{ik}\right)  ^{2}\right)\right\rbrace =c h_{k},\quad 1\leq k\leq s.
\end{align*}

If we consider an invariant metric satisfying $g_{ik}=f_{ik}=1$ and $h_k=h$, then the previous algebraic system reduces to following one
%
%

\begin{eqnarray*}
4m+2(m+1)(2-h)+2m(n-2m)&=&8(n+1) c,\\
4(m+1)+2m(n-m)h^2&=&8(n+1)ch.
\end{eqnarray*}
In this way, we get that

\begin{eqnarray*}
h &=& \frac{2(m+1)+(n-2m+2)m \pm \sqrt{\Delta}}{2\left[(n-m)m+m+1\right]},\\ \\
c &=&\dfrac{4m(n-2m+2)[mn-m^2+m+1]+(m+1)[6mn-4m^2+4]\mp 2(m+1)\sqrt{\Delta}}{16(n+1)\left[(n-m)m+m+1\right]},
\end{eqnarray*}
where $\Delta=m^2n^2-4(m^3+m)n+(4m^4-8m^3+8m^2-4)$. 
Note that $\Delta\geq 0$ if 
$$
n m \geq 2 \left[ m^2+1+\sqrt{2(m^3+1)}\right] \geq 8
$$
since $m\geq 1$. It is easy to see that if $n>2m$ then $h>0$. Besides, these metrics are non isometric since $c_1\neq c_2$. 

If $n=m$ we obtain the isotropy irreducible space $Sp(n)/U(n)$ that (up to homotheties) admits a unique invariant metric which is Einstein, ( see 7.44, \cite{Besse}).  

For $m=1$ we obtain the Sakane's result \cite{Sakane}, which provides the invariant Einstein metrics $f=g=1, h=\dfrac{4+n\pm\sqrt{(n-8)n}}{2(n+1)}$ with $c=\dfrac{4n+n^2\mp \sqrt{(n-8)n}}{4(n+1)^2}$ on the full flag manifold $Sp(n)/ U\left(1\right)^n $. \qed

\begin{example}
If we fix $m=2$, then for each $n\geq 10$ the flag manifold $Sp(n)/ U\left(
2\right)^s$, $n=2s$,  admits at least two non K\"{a}hler (and non isometric) invariant Einstein metrics
\begin{eqnarray*}
1)\quad f &=& g=1, \quad h = \frac{n+1 + \sqrt{(n-5)^2-18}}{2n-1}\\ \\
2)\quad f &=& g=1, \quad h = \frac{n+1 - \sqrt{(n-5)^2-18}}{2n-1}.
\end{eqnarray*}

\end{example}
\begin{corollary}\label{eq full flag Cn}
The Einstein equations on the full flag manifold $Sp(n)/ U\left(1\right)^n $ reduce to an algebraic system of $n^2$ equations and unknowns $g_{ij}$, $f_{ij}$, $h_i$:

\begin{align*}
&4+\frac{2 }{h_{k}f_{kt}}\left(
g_{kt}^{2}-\left( h_{k}-f_{kt}\right)  ^{2}\right)+\frac{2  }{h_{t}f_{kt}}\left(  g_{kt}^{2}-\left(  h_{t}-f_{kt}\right)  ^{2}\right)\\ \\
&+\sum_{i\neq k,t}^{n}\frac{1}{g_{ik}g_{it}
}\left(  g_{kt}^{2}-\left(  g_{ik}-g_{it}\right)  ^{2}\right)+
\sum_{i\neq k,t}^{n}\frac{1}{f_{ik}f_{it}
}\left(  g_{kt}^{2}-\left(  f_{ik}-f_{it}\right)  ^{2}\right)
=g_{kt},\,\quad 1\leq k\neq t\leq n;
\end{align*}

\begin{align*}
&4+\frac{2 }{h_{k}g_{kt}}\left(
f_{kt}^{2}-\left( h_{k}-g_{kt}\right)  ^{2}\right)
+\frac{2  }{h_{t}g_{kt}%
}\left(  f_{kt}^{2}-\left(  h_{t}-g_{kt}\right)  ^{2}\right)\\ \\
&+\sum_{i\neq k,t}^{n}\frac{1}{f_{it}g_{ik}%
}\left(  f_{kt}^{2}-\left(  f_{it}-g_{ik}\right)  ^{2}\right)
+\sum_{i\neq k,t}^{n}\frac{1}{f_{ik}g_{it}
}\left(  f_{kt}^{2}-\left(  f_{ik}-g_{it}\right)  ^{2}\right)
=f_{kt},\quad 1\leq k\neq t\leq n;\\ \\
&8+2\sum_{i\neq k}^{n}\frac{1}{f_{ik}g_{ik}%
}\left(  h_{k}^{2}-\left(  f_{ik}-g_{ik}\right)  ^{2}\right)=h_{k},\quad 1\leq k\leq n.
\end{align*}

\end{corollary}

\section{Proof of Theorem B}

Now we consider the homogeneous spaces of the form $
SO(2n)/U(n_{1})\times\cdots\times U(n_{s})$, where $n\geq 4$ and $n=\sum n_{i}$. We see the Lie algebra $\mathfrak{so}\left( 2n,\mathbb{C}\right)$ as the algebra of the skew-symmetric matrices in even dimension. These matrices can be written as
\[
A=\left(\begin{array}{cc}
\alpha & \beta \\
\gamma & -\alpha^t
\end{array} \right)
\]
where $\alpha,\beta,\gamma$ are matrices $n\times n$ with $\beta,\gamma$ skew-symmetric.
A Cartan subalgebra of $\mathfrak{so}\left( 2n,\mathbb{C}\right)$
consists of matrices of the form 
\begin{equation}\label{Cartan Dn}
\mathfrak{h}=\{  \mathrm{diag}\left(
\varepsilon_{1},\ldots,\varepsilon_{n},-\varepsilon_{1},\ldots,-\varepsilon_{n}\right) ;\,\varepsilon_{i}%
\in\mathbb{C}\}  .
\end{equation}
The root system of the pair  $\left( \mathfrak{so}\left( 2n,\mathbb{C}\right),\mathfrak{h}\right)$ is given by
\begin{equation}\label{root system Dn}
R=\{  \pm\left(  \varepsilon_{i}\pm\varepsilon_{j}\right);\,1\leq i<j\leq n\}.
\end{equation}
The root system for the subalgebra
$\mathfrak{k}_{\Theta}^{\mathbb{C}}=\mathfrak{sl}\left(
n_{1},\mathbb{C}\right) \times\cdots\times\mathfrak{sl}\left(
n_{s},\mathbb{C}\right)$ is
$$
R_{\Theta}=\{  \pm\left(
\varepsilon_{c}^{i}-\varepsilon_{d}^{i}\right) ;\,1\leq c<d\leq
n_{i}\}  ,
$$
then
\[
R_{M}^{+}=\{
\varepsilon_{a}^{i}\pm\varepsilon_{b}^{j};\,1\leq i<j\leq s\}
\cup\{  \varepsilon_{a}^{i}+\varepsilon_{b}^{i};\,a<b\} .
\]

The subalgebra $\mf{t}$ is formed by matrices of the form
\[
\mathfrak{t}=\{\mathrm{diag}\left(
\varepsilon_{n_{1}}^{1},\ldots,\varepsilon_{n_{1}}^{1},\ldots,\varepsilon_{n_{s}}^{s},\ldots,\varepsilon_{n_{s}}^{s},-\varepsilon_{n_{1}}^{1},\ldots
,-\varepsilon_{n_{1}}^{1},\ldots,-\varepsilon_{n_{s}}^{s},\ldots,-\varepsilon_{n_{s}}^{s}\right)\in i\mathfrak{h}_{\mathbb{R}}\}
\]
where $\varepsilon_{n_{i}}^{i}$ appears exactly  $n_{i}$ times. By
restricting the roots of $R_{M}^{+}$ to $\mathfrak{t}$, and
using the notation $\delta_{i}=k(\varepsilon_{a}^{i})$, we obtain
the set of positive t-root:
\[
R^{+}_{\mathfrak{t}}=\{\delta_{i}\pm\delta_{j},2\delta_{i};1\leq
i<j\leq s\}.
\]
In particular there exist $s^{2}$ positive t-roots.

The Killing form on  $\mathfrak{so}\left(
2n\right)$ is given by  $B\left(  X,Y\right)  =2\left(
n-1\right)  trXY$ and $B\left(  \alpha ,\alpha\right)
=\frac{1}{2\left(  n-1\right)},$ for all
$\alpha\in R$. The eigenvectors $X_{\alpha}$ satisfying (\ref{base weyl}) are given by
$$
E_{ab}^{ij}=\frac{1}{2\sqrt{n-1}}E_{\varepsilon_{a}^{i}-\varepsilon_{b}^{j}
},\quad F_{ab}^{ij}=\frac{1}{2\sqrt{n-1}}E_{\varepsilon_{a}^{i}
+\varepsilon_{b}^{j}},
$$
$$
 F_{-ab}^{ij}=\frac{1}{2\sqrt{n-1}}E_{-(\varepsilon_{a}^{i}
+\varepsilon_{b}^{j})},\quad G_{ab}^{i}=\frac{1}{2\sqrt{n-1}}E_{\varepsilon
_{a}^{i}+\varepsilon_{b}^{i}}
$$
where $E_\alpha$ denotes the canonical eigenvector of $\mf{g}_{\alpha}$.
The non zero square of structures constants is  $N_{\alpha,\beta}^2=
1/4(n-1).$

The notation for the invariant scalar product on the base $\{ X_{\alpha}; \alpha\in R_{M}\}$ is given by
\begin{equation}\label{notation Dn}
g_{ij}=g(  E_{ab}^{ij},E_{ba}^{ji}),\quad f_{ij}=g(  F_{ab}^{ij}, F_{-ab}^{ij}),\quad h_{i}=g( G_{ab}^{i},G_{ba}^{i}),
\end{equation}
with $1\leq i<j\leq s$.
According to \cite{Arv art}, the Einstein equations on the spaces
$SO(2n)/U(n_{1})\times\cdots\times U(n_{s})$ reduce to an algebraic system of $s^{2}$ equations and $s^{2}$ unknowns $g_{ij},$ $f_{ij},h_{i}:$%

\[
 n_{i}+n_{j}+\frac{1}{2}\left\lbrace   \sum_{l\neq i,j}\frac{n_{l}}{g_{il}g_{jl}%
}\left(  g_{ij}^{2}-\left(  g_{il}-g_{jl}\right)  ^{2}\right)
+\sum_{l\neq i,j}\frac{n_{l}}{f_{il}f_{jl}}\left(
g_{ij}^{2}-\left(  f_{il}-f_{jl}\right) ^{2}\right)\right.  
\]%
\[
 + \left. \frac{n_{i}-1}{f_{ij}h_{i}}\left(  g_{ij}^{2}-\left(  f_{ij}%
-h_{i}\right)  ^{2}\right)  +\frac{n_{j}-1}{f_{ij}h_{j}}\left(  g_{ij}%
^{2}-\left(  f_{ij}-h_{j}\right)  ^{2}\right)  \right\rbrace   =4(n-1)c g_{ij},
\]

\[
  n_{i}+n_{j}+\frac{1}{2}\left\lbrace   \sum_{l\neq i,j}\frac{n_{l}}{g_{il}f_{jl}%
}\left(  f_{ij}^{2}-\left(  g_{il}-f_{jl}\right)  ^{2}\right)
+\sum_{l\neq i,j}\frac{n_{l}}{f_{il}g_{jl}}\left(
f_{ij}^{2}-\left(  f_{il}-g_{jl}\right) ^{2}\right)\right. 
\]%
\[
 +\left. \frac{n_{i}-1}{g_{ij}h_{i}}\left(  f_{ij}^{2}-\left(  g_{ij}%
-h_{i}\right)  ^{2}\right)  +\frac{n_{j}-1}{g_{ij}h_{j}}\left(  f_{ij}%
^{2}-\left(  g_{ij}-h_{j}\right)  ^{2}\right)   \right\rbrace  =4(n-1)c f_{ij},
\]

\[
  2\left(  n_{i}-1\right)  +\sum_{l\neq
i}\frac{n_{l}}{g_{il}f_{il}}\left( h_{i}^{2}-\left(
g_{il}-f_{il}\right)  ^{2}\right)   =4(n-1)c h_{i}.
\]


If we consider the invariant metric $g_{ij}=g$, $f_{ij}=f$ and $h_i=h$ on the space $SO(2n)/U(m)^s$, $m>1$, the Einstein equation reduce to the following algebraic system
\begin{eqnarray*}
 2m+\dfrac{1}{2}\left[m(s-2)+m\dfrac{g^2}{f^2}(s-2)\right]+\dfrac{m-1}{fh}(g^2-(f-h)^2)   &=&4(n-1)cg\\ \\
2m+ \dfrac{m}{gf}(s-2)(f^2-(g-f)^2)+\dfrac{m-1}{gh}(f^2-(g-h)^2)  &=&4(n-1)cf \\ \\
2(m-1)+\dfrac{m}{gf}(s-1)(h^2-(g-f)^2)&=&4(n-1)ch.
\end{eqnarray*}

In a similiar way, as in the case $C_n$, if $f=g=1$ we obtain 

\begin{eqnarray*}
 n+(m-1)(2-h) &=&4(n-1)c\\ \\
2(m-1)+m(s-1)h^2&=&4(n-1)ch.
\end{eqnarray*}
By solving explicitly this algebraic system one obtains the two non isometric metric of Theorem B.  \qed

The previous result does not applies to the full flag manifold $SO(2n)/U(1)^n$, because on this space any invariant metric does not depend on the parameter $h_i$, it is determined only by positive scalars $g_{ij}$ and $f_{ij}$. This case was treated in \cite{Sakane}.

\end{document}